\documentclass[a4paper,oneside,reqno]{amsart}
\usepackage[T1]{fontenc}
\usepackage[utf8]{inputenc}
\usepackage{standalone}
\usepackage{lmodern}
\usepackage{textcomp}
\usepackage{gensymb}
\usepackage{microtype}
\usepackage{ellipsis}
\usepackage{amsmath}
\usepackage{amssymb}
\usepackage[hyperindex,breaklinks]{hyperref}
\usepackage{amsthm}
\usepackage{cleveref}

\newtheorem{theorem}{Theorem}[section]
\newtheorem*{theorem*}{Theorem}
\newtheorem{proposition}[theorem]{Proposition}
\newtheorem{lemma}[theorem]{Lemma}
\theoremstyle{definition}
\newtheorem{definition}[theorem]{Definition}
\newtheorem{remark}[theorem]{Remark}
\theoremstyle{remark}
\newtheorem{example}[theorem]{Example}

\numberwithin{equation}{section}
\numberwithin{figure}{section}

\usepackage{mathtools}
\usepackage[backend=bibtex,maxbibnames=99]{biblatex}
\newcommand\R{\mathbb R}
\newcommand\C{\mathbb C}
\newcommand\F{\mathbb F}

\DeclareMathOperator\Tr{Tr}
\DeclareMathOperator\aff{aff}

\usepackage{amsmath}
\usepackage{tikz}
\usetikzlibrary{backgrounds}
\usetikzlibrary{calc,arrows}
\usetikzlibrary{decorations}
\usetikzlibrary{decorations.markings}
\usetikzlibrary{decorations.pathmorphing,shapes}
\usetikzlibrary{arrows}
\usetikzlibrary{patterns}

\pgfdeclaredecoration{triangle}{start}{
  \state{start}[width=0.99\pgfdecoratedinputsegmentremainingdistance,next state=up from center]
  {\pgfpathlineto{\pgfpointorigin}}
  \state{up from center}[next state=do nothing]
  {
    \pgfpathlineto{\pgfqpoint{\pgfdecoratedinputsegmentremainingdistance}{\pgfdecorationsegmentamplitude}}
    \pgfpathlineto{\pgfqpoint{\pgfdecoratedinputsegmentremainingdistance}{-\pgfdecorationsegmentamplitude}}
    \pgfpathlineto{\pgfpointdecoratedpathfirst}
  }
  \state{do nothing}[width=\pgfdecorationsegmentlength,next state=do nothing]{
    \pgfpathlineto{\pgfpointdecoratedinputsegmentfirst}
    \pgfpathmoveto{\pgfpointdecoratedinputsegmentlast}
  }
}

\tikzset{
    triangle path/.style={decoration={triangle,amplitude=#1}, decorate},
    triangle path/.default=1ex}

\begin{document}

\title[Polytopes of eigensteps of finite equal norm tight frames]{Polytopes of eigensteps\\of finite equal norm tight frames}

\author{Tim Haga}
\address{Department of Mathematics, University of Bremen, 28359 Bremen, Germany}
\email{timhaga@math.uni-bremen.de}

\author{Christoph Pegel}
\address{Department of Mathematics, University of Bremen, 28359 Bremen, Germany}
\email{pegel@math.uni-bremen.de}

\begin{abstract}
 Hilbert space frames generalize orthonormal bases to allow redundancy in representations of vectors while keeping good reconstruction properties. A frame comes with an associated frame operator encoding essential properties of the frame. We study a polytope that arises in an algorithm for constructing all finite frames with given lengths of frame vectors and spectrum of the frame operator, which is a Gelfand-Tsetlin polytope. For equal norm tight frames, we give a non-redundant description of the polytope in terms of equations and inequalities. From this we obtain the dimension and number of facets of the polytope.
While studying the polytope, we find two affine isomorphisms and show how they relate to operations on the underlying frames.
\end{abstract}
\maketitle

\section{Introduction}

Eigensteps have been introduced by Cahill, Fickus, Mixon, Poteet and Strawn in \cite{CFM11} to construct all finite frames of a given spectrum and set of lengths. The results have been adopted in \cite{FMP11} to obtain an algorithm to construct all self-adjoint matrices with prescribed spectrum and diagonal. The existence of such matrices is given by the Schur-Horn~Theorem. The fact that eigensteps form a polytope, and therefore a path-connected set, has been used in \cite{CMS13} to obtain connectivity and irreducibility results for algebraic varieties of finite unit norm tight frames. Parametrizing this polytope is crucial to apply the algorithms described in \cite{CFM11} and \cite{FMP11}.

In this paper, we consider the case of equal norm tight frames, where the describing equations and inequalities of the polytope of eigensteps can be drastically simplified. To be precise, we give a description of the polytope where the remaining inequalities are in one-to-one correspondence with the facets of the polytope and the remaining equations are linearly independent.

We start with the necessary preliminaries in \Cref{sec:preliminaries} in order to study the polytope of eigensteps in a purely combinatoric manner in \Cref{sec:dimension,sec:facets}. We give formulae for the dimension of the polytope and its number of facets:
\begin{theorem*}
    Let $\Lambda_{N,d}$ be the polytope of eigensteps of equal norm tight frames of $N$ vectors in a $d$-dimensional Hilbert space.
    \begin{enumerate}
        \item The dimension of $\Lambda_{N,d}$ is $0$ for $d=0$ and $d=N$, otherwise
            \begin{equation*}
                \dim(\Lambda_{N,d})=(d-1)(N-d-1).
            \end{equation*}
        \item For $2\le d\le N-2$ the number of facets of $\Lambda_{N,d}$ is
            \begin{equation*}
                d(N-d-1)+(N-d)(d-1)-2.
            \end{equation*}
    \end{enumerate}
\end{theorem*}
This theorem appears as \Cref{thm:dimension} and \Cref{thm:nfacets}, respectively.
In \Cref{sec:frame-connections} we return to frame theory and describe how the affine isomorphisms of polytopes we obtained combinatorially are described by reversing the order of frame vectors and taking Naimark complements. We end with \Cref{sec:conclusion}, where we discuss our results and some open questions.

\section{Preliminaries}
\label{sec:preliminaries}

Given a finite dimensional real or complex Hilbert space $\mathcal H$ of dimension $d$, a frame is just a spanning set $F=\left(f_n\right)_{n=1}^N$ of $\mathcal H$. By a slight abuse of notation, we identify $F$ with the $d\times N$ matrix having as columns the coordinates of the frame vectors $f_1,\ldots,f_N$ with respect to some orthonormal basis of $\mathcal H$. Since any finite dimensional Hilbert space is isomorphic to $\R^d$ or $\C^d$ by a choice of an orthonormal basis, we assume $\mathcal H = \F^d$ where $\F=\C$ or $\R$ and use coordinates with respect to the standard basis. A frame $F=\left(f_n\right)_{n=1}^N$ comes with an associated \emph{frame operator} $T\colon \mathcal H\to\mathcal H$ given by $Tv = \sum_{n=1}^N \langle v, f_n\rangle f_n$. Let $F^*$ denote the conjugate transpose of the matrix $F$, then the frame operator is given by $FF^*$. A frame is called \emph{equal norm} if $\|f_n\|^2=\mu$ is the same for all frame vectors, \emph{tight} if its frame operator is a multiple of the identity, and \emph{Parseval} if its frame operator is equal to the identity. When $F$ is a finite equal norm tight frame, we have $FF^* = \frac{N\mu}{d} \cdot I_d$, where $I_d$ is the $d\times d$ identity matrix. We refer to \cite{CK13} for a detailed introduction and collection of recent results in finite frame theory.

The problem discussed in \cite{CFM11} is the following: given a non-increasing sequence of norm-squares $(\mu_n)_{n=1}^N$ and a non-increasing, non-negative spectrum $(\lambda_i)_{i=1}^d$, find all matrices $F=(f_n)_{n=1}^N$ such that $\|f_n\|^2 = \mu_n$ for all $n$ and $\sigma(FF^*)=(\lambda_i)_{i=1}^d$, where $\sigma$ denotes the non-increasing spectrum of an operator. To achieve this, the authors of \cite{CFM11} divide the task into two steps. First, find all possible sequences of spectra $((\lambda_{i,n})_{i=1}^d)_{n=0}^N$, such that there exists an $F$ with $\|f_n\|^2=\mu_n$ and $\sigma(F_n^{\vphantom{*}} F_n^*)=(\lambda_{i,n})_{i=1}^d$ for all $n$, where $F_n$ is $F$ truncated to the first $n$ columns. Any such sequence of spectra is called a \emph{valid sequence of eigensteps} for the given input data $(\mu_n)_{n=1}^N$ and $(\lambda_i)_{i=1}^d$. Then, for a given valid sequence of eigensteps, find all $F$ such that $\|f_n\|^2=\mu_n$ and $\sigma(F_n^{\vphantom{*}} F_n^*)=(\lambda_{i,n})_{i=1}^d$ for all $n$ by iteratively adding frame vectors following an elaborate algorithm.

Since $F_{n+1}^{\vphantom{*}} F_{n+1}^* = F_n^{\vphantom{*}} F_n^* + f_{n+1}^{\vphantom{*}} f_{n+1}^*$, a theorem by Horn and Johnson \cite[Section~4.3]{HJ85} states that the spectra of $F_n^{\vphantom{*}} F_n^*$ and $F_{n+1}^{\vphantom{*}} F_{n+1}^*$ \emph{interlace}. That is, when spectra are indexed in non-increasing order, we have
\begin{equation}
    \lambda_{d,n} \le \lambda_{d,n+1}
    \le
    \lambda_{d-1,n} \le \lambda_{d-1,n+1}
    \le \cdots \le
    \lambda_{2,n} \le \lambda_{2,n+1}
    \le
    \lambda_{1,n} \le \lambda_{1,n+1}.
    \label{eq:interlace}
\end{equation}
Furthermore, for $0\le n \le N$ we have
\begin{equation}
    \sum_{i=1}^d \lambda_{i,n} = \Tr(F_n^{\vphantom{*}} F_n^*) = \Tr(F_n^* F_n^{\vphantom{*}}) = \sum_{k=1}^n \|f_k\|^2 = \sum_{k=1}^n \mu_k.
    \label{eq:sums}
\end{equation}
The second equality in \eqref{eq:sums} follows from the invariance of the trace under cyclic permutations.

By Theorem~2 in \cite{CFM11}, conditions \eqref{eq:interlace} and \eqref{eq:sums} together with $\lambda_{i,0} = 0$ and $\lambda_{i,N} = \lambda_i$ for all $i$ completely characterize the valid sequences of eigensteps. Since all conditions are linear equations or linear inequalities, the valid sequences of eigensteps form a polytope $\Lambda((\mu_n)_{n=1}^N, (\lambda_i)_{i=1}^d)$ in $\R^{d\times (N+1)}$.

Note that this polytope coincides with the Gelfand-Tsetlin polytope introduced in~\cite{GT50}.
The corresponding polytope of Gelfand-Tsetlin patterns is obtained by padding the sequence $(\lambda_{i})_{i=1}^d$ with $N-d$ zeros and using it as the top row.
The conditions in \eqref{eq:sums} yield row sums of the Gelfand-Tsetlin patterns, see~\cite[Def.~1.2]{DM04}. 
However, we do not assume the sequences $(\mu_n)_{n=1}^N$ and $(\lambda_i)_{i=1}^d$ to be integral in general.

In this paper, we consider the case of equal norm tight frames, where $\mu_n = \mu$ for all $n$ and $FF^*=\frac{N\mu}{d} \cdot I_d$. In particular, this covers equal norm Parseval frames for $\mu=\frac{d}{N}$ and unit norm tight frames for $\mu=1$. To avoid fractions and increase readability, we discuss equal norm tight frames with $\mu=d$, hence $FF^* = N \cdot I_d$. By scaling, the results can of course be transferred to arbitrary finite equal norm tight frames.

Let $\Lambda_{N,d} \coloneqq \Lambda( (d)_{n=1}^N, (N)_{i=1}^d)$ denote the polytope of finite equal norm tight frames of size $N$ with norm-squares $\mu=d$. We arrive at the following combinatorial definition of the polytope of eigensteps:
\begin{definition} \label{def:eigensteps}
    For integers $0\le d\le N$, we define the polytope $\Lambda_{N,d}$ as the set of all matrices
    \begin{equation*}
        \lambda = \big( \lambda_{i,n} \big)_{\substack{1\le i\le d,\\0\le n\le N}} \in \R^{d\times(N+1)}
    \end{equation*}
    satisfying the following conditions:
    \begin{align}
        \vphantom{\sum_{i=1}^d}
        \lambda_{i,0} &= 0
        & \text{for} \quad & 1\le i \le d, \label{eq:def-0} \\
        \lambda_{i,N} &= N
        & \text{for} \quad & 1\le i \le d, \label{eq:def-N} \\
        \sum_{i=1}^d \lambda_{i,n} &= dn
        & \text{for} \quad & 0\le n \le N, \label{eq:def-sums} \\
        \lambda_{i,n} &\le \lambda_{i,n+1}
        & \text{for} \quad & 1\le i \le d,\enskip 0\le n < N, \label{ineq:def-horiz} \\
        \vphantom{\sum_{i=1}^d}
        \lambda_{i,n} &\le \lambda_{i-1,n-1}
        & \text{for} \quad & 1< i\le d,\enskip 0 < n \le N. \label{ineq:def-diag}
    \end{align}
    We will refer to \eqref{eq:def-0} and \eqref{eq:def-N} as the \emph{first} and \emph{last column conditions}, respectively. The equations in \eqref{eq:def-sums} are \emph{column sum conditions}, while \eqref{ineq:def-horiz} and \eqref{ineq:def-diag} will be referred to as the \emph{horizontal} and \emph{diagonal inequalities}, respectively, for reasons obvious from \Cref{fig:eigensteps}.
\end{definition}

\begin{figure}
    \centering
    \includegraphics[width=\textwidth]{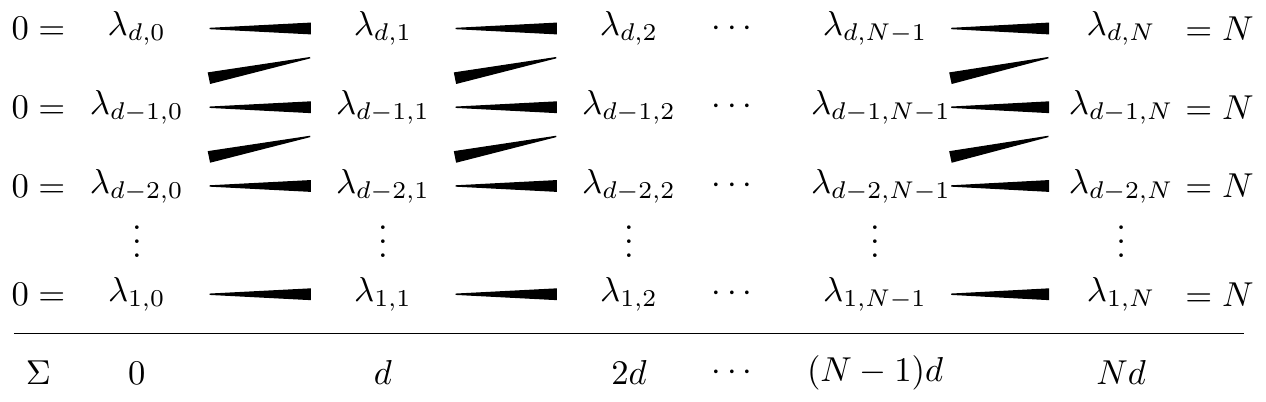}
    \caption[]{The conditions for a valid sequence of eigensteps for equal norm tight frames with $\mu=d$. A wedge%
       \tikz[baseline=-\the\dimexpr\fontdimen22\textfont2\relax]{
           \node (b) at (1.5,0) {$\strut\lambda_{k,l}$};
           \node[baseline=(b.base)] (a) at (0,0) {$\strut\lambda_{i,j}$};
           \draw [fill, triangle path=1.5pt] (a) -- (b);
       }%
   denotes an inequality $\lambda_{i,j}\le\lambda_{k,l}$.}
    \label{fig:eigensteps}
\end{figure}

\section{The dimension of polytopes of eigensteps of finite equal norm tight frames}\label{sec:dimension}

In this section we determine the dimension of $\Lambda_{N,d}$. The dimension of the solution set of a system of linear equations and inequalities can be computed from the number of variables and the number of linearly independent equations, including those arising from inequalities that are always satisfied with equality. Thus, the first step is to remove redundant equations and recognize inequalities that are always satisfied with equality.

\begin{proposition}\label{prop:eigensteps}
    A matrix $(\lambda_{i,n})\in\R^{d\times(N+1)}$ is a point of $\Lambda_{N,d}$ if and only if the following conditions are satisfied:
    \begin{align}
        \vphantom{\sum_{i=1}^d}
        \lambda_{i,n} &= 0
        & \text{for} \quad & i > n, \label{eq:prop1-0} \\
        \lambda_{i,n} &= N
        & \text{for} \quad & i < n+d-N+1, \label{eq:prop1-N} \\
        \sum_{i=1}^d \lambda_{i,n} &= dn
        & \text{for} \quad & 0< n < N, \label{eq:prop1-sums} \\
        \lambda_{i,n} &\le \lambda_{i,n+1}
        & \text{for} \quad & 1\le i \le d,\enskip i\le n < N-d+i-1, \label{ineq:prop1-horiz} \\
        \vphantom{\sum_{i=1}^d}
        \lambda_{i,n} &\le \lambda_{i-1,n-1}
        & \text{for} \quad & 1< i\le d,\enskip i \le n < N-d+i, \label{ineq:prop1-diag} \\
        \lambda_{d,d} &\ge 0, \label{ineq:prop1-lb} \\
        \vphantom{\sum_{i=1}^d}
        \lambda_{1,N-d} &\le N. \label{ineq:prop1-ub}
    \end{align}
\end{proposition}

\begin{proof}
    The idea behind the proof is to use the first and last column conditions together with the horizontal and diagonal inequalities to obtain triangles in the eigenstep tableaux that consist of fixed $0$- or $N$-entries. Using those fixed triangles we can drop many of the now redundant inequalities from the system in \Cref{def:eigensteps}. The remaining inequalities form a parallelogram with two legs as depicted in \Cref{fig:eigenstep-triangles}.

    We first prove the necessity of the modified conditions. The triangles described by \eqref{eq:prop1-0} and \eqref{eq:prop1-N}---from now on referred to as the two \emph{triangle conditions}, see \Cref{fig:eigenstep-triangles} for reference---are an immediate consequence of the first and last column conditions together with the horizontal and diagonal inequalities. The remaining equations and inequalities already appear as part of the definition of $\Lambda_{N,d}$.

    To prove sufficiency, we first see that the first and last column conditions are implied by the triangle conditions. The first and last column are always fixed, so the column sum conditions can be weakened to \eqref{eq:prop1-sums}. Condition \eqref{ineq:prop1-lb} together with the weakened horizontal and diagonal inequalities \eqref{ineq:prop1-horiz} and \eqref{ineq:prop1-diag} is enough to guarantee that all $\lambda_{i,n}$ are non-negative. Thus, we will refer to \eqref{ineq:prop1-lb} as the \emph{lower bound condition}. Similarly \eqref{ineq:prop1-ub} guarantees $\lambda_{i,n}\le N$ for all entries and will be referred to as the \emph{upper bound condition}. Hence, from the original horizontal and diagonal inequalities \eqref{ineq:def-horiz} and \eqref{ineq:def-diag} we only need those involving solely entries outside of the $0$- and $N$-triangles.
\end{proof}
The remaining inequalities required by \Cref{prop:eigensteps} are  depicted in \Cref{fig:eigenstep-triangles}.
Note that \Cref{prop:eigensteps} holds only for equal norm tight frames, in particular \eqref{eq:prop1-N} is false for frames
which are not tight. 

\begin{figure}
    \centering
    \includegraphics{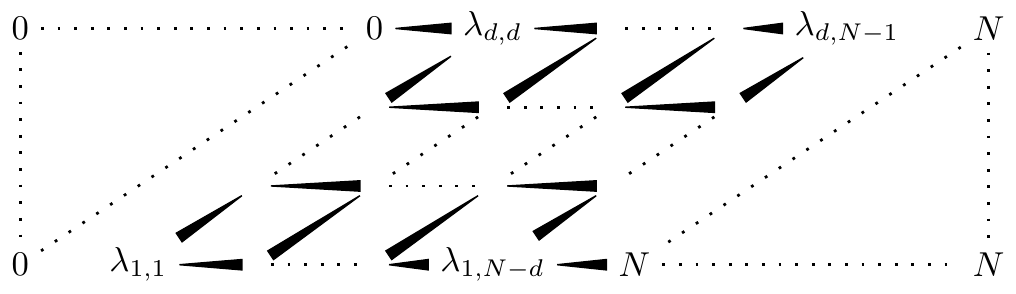}
    \caption{The modified conditions for a valid sequence of eigensteps for equal norm tight frames with only the inequalities required by \Cref{prop:eigensteps}.}
    \label{fig:eigenstep-triangles}
\end{figure}

With the modified conditions from \Cref{prop:eigensteps} we are now able to compute the dimension of $\Lambda_{N,d}$.

\begin{theorem} \label{thm:dimension}
    The dimension of the polytope $\Lambda_{N,d}$ is $0$ for $d=0$ and $d=N$, otherwise \[\dim(\Lambda_{N,d})=(d-1)(N-d-1).\]
\end{theorem}

\begin{proof}
    For $d=0$ the only point in $\Lambda_{N,d}$ is the empty $0\times (N+1)$ matrix, hence $\dim(\Lambda_{N,0})=0$. For $d=N$, the $0$- and $N$-triangles fill up the whole matrix. Thus, $\Lambda_{N,N}$ also consists of a single point and $\dim(\Lambda_{N,N})=0$.
    
    Otherwise, the triangle and sum conditions given by \eqref{eq:prop1-0}, \eqref{eq:prop1-N} and \eqref{eq:prop1-sums} are linearly independent. Thus, by counting the equations, we obtain
    \begin{align*}
        \dim(\Lambda_{N,d}) &\le d(N+1) - 2\cdot \frac{d(d+1)}2 - (N-1)
        = (d-1)(N-d-1).
    \end{align*}
    To verify $\dim(\Lambda_{N,d}) \ge (d-1)(N-d-1)$, we show that $\Lambda_{N,d}$ contains a \emph{special point} $\widehat\lambda$ that satisfies all the inequalities \eqref{ineq:prop1-horiz} to \eqref{ineq:prop1-ub} \emph{strictly}, with the difference between the left and right hand sides of each inequality being equal to $1$. The entries of $\widehat\lambda$ not fixed by the triangle conditions are given by
    \begin{equation}
        \label{eq:lambdahat}
        \widehat\lambda_{i,n} \coloneqq d+n-2i+1 \quad\text{for}\quad i\le n\le N-d+i-1.
    \end{equation}
    See \Cref{ex:lambdahat} for reference. The smallest value in \eqref{eq:lambdahat} is $\widehat\lambda_{d,d}=1$, the largest is $\widehat\lambda_{1,N-d}=N-1$, so the lower and upper bound conditions are strictly satisfied. The horizontal and diagonal inequalities hold strictly as well, since
    \begin{align*}
        \widehat\lambda_{i,n} = d+n-2i+1 &< d+(n+1)-2i+1 = \widehat\lambda_{i,n+1}, \\
        \widehat\lambda_{i,n} = d+n+2i+1 &< d+(n-1)-2(i-1)+1 = \widehat\lambda_{i-1,n-1}.
    \end{align*}
    It remains to verify the column sum conditions \eqref{eq:prop1-sums}. Letting $i_0\coloneqq\max\{0,n+d-N\}$ and $i_1\coloneqq\min\{d,n\}$ we have
    \begin{align*}
        \sum_{i=1}^d \widehat\lambda_{i,n} &= \sum_{i=1}^{i_0} N + \sum_{i=i_0+1}^{i_1} \widehat\lambda_{i,n} + \sum_{i=i_1+1}^d 0 \\
       &= i_0 N + \sum_{i=i_0+1}^{i_1} (d+n-2i+1) \\
       &= i_0 N + (i_1-i_0)(d+n-i_1-i_0).
   \end{align*}
   In all four cases, this expression evaluates to $dn$. %, to be precise
\end{proof}

\begin{example}    
    \label{ex:lambdahat}
    For $N=6$, $d=4$ we obtain the special point of $\Lambda_{6,4}$ as
    \begin{equation*}
        \widehat\lambda =
        \begin{pmatrix}
            0 & 0 & 0 & 0 & 1 & 2 & 6 \\
            0 & 0 & 0 & 2 & 3 & 6 & 6 \\
            0 & 0 & 3 & 4 & 6 & 6 & 6 \\
            0 & 4 & 5 & 6 & 6 & 6 & 6
        \end{pmatrix}.
    \end{equation*}
    This tableau satisfies all inequalities in \Cref{prop:eigensteps} strictly while also satisfying the column sum and triangle conditions.
\end{example}

Note that the dimension of the polytope of eigensteps $\Lambda_{N,d}$ is related to the dimensions of certain frame varieties.

\begin{remark}
	Let $\mathcal{F}_{N,d}^\R\subseteq\R^{d\times N}$ be the real algebraic variety of real unit norm tight frames, $\mathcal{F}_{N,d}^\C\subseteq\C^{d\times N} = \R^{2(d\times N)}$ the real algebraic variety of complex unit norm tight frames.
	The orthogonal group $O(d)$ and the unitary group $U(d)$ act on $\mathcal{F}_{N,d}^\R$ and $\mathcal{F}_{N,d}^\C$, respectively.
	The dimensions of $\mathcal{F}_{N,d}^\R$ and $\mathcal{F}_{N,d}^\C$ as determined in~\cite[Prop.~5.5]{CMS13} are strictly greater than $\dim(\Lambda_{N,d})$ for $N,d>0$.
	By Theorem~4.3 in~\cite{DS06}, this is also true for the real dimension of $\mathcal{F}_{N,d}^\C/U(d)$, while the dimension of $\mathcal{F}_{N,d}^\R/O(d)$ is in fact equal to $\dim(\Lambda_{N,d})$.
\end{remark}

\section{The facets of polytopes of finite equal norm tight frames}
\label{sec:facets}
In this section we investigate which of the remaining inequalities describing $\Lambda_{N,d}$ are necessary. In other words, we find the facet-describing inequalities of $\Lambda_{N,d}$. In particular, we obtain a formula for the number of facets.

To reduce the number of inequalities we need to consider separately, we use two kinds of dualities. One is an affine isomorphism between $\Lambda_{N,d}$ and $\Lambda_{N,N-d}$ that translates horizontal to diagonal inequalities and vice versa. The other is an affine involution on $\Lambda_{N,d}$, reversing the order of rows and columns of the eigenstep tableaux. We will see in \Cref{sec:frame-connections} how these dualities correspond to certain operations on equal norm tight frames.

From the proof of \Cref{thm:dimension} we know that the affine hull $\aff(\Lambda_{N,d})$ is the affine subspace of $\R^{d\times(N+1)}$ defined by the triangle and sum conditions (\eqref{eq:prop1-0}, \eqref{eq:prop1-N} and \eqref{eq:prop1-sums}).

\begin{proposition} \label{prop:psi}
    There is an affine isomorphism 
    \begin{equation*}
	    \Psi_{N,d}\colon \aff(\Lambda_{N,d}) \longrightarrow \aff(\Lambda_{N,N-d})
    \end{equation*}
    given by
    \begin{equation*}
        \left(\Psi_{N,d}(\lambda)\right)_{i,n} =
        \begin{cases}
            \lambda_{d+i-n, N-n}, & \text{for } i \le n \le d+i-1,\\
            0, & \text{for } n<i, \\
            N, & \text{for } n>d+i-1,
        \end{cases}
    \end{equation*}
    that restricts to an affine isomorphism $\Lambda_{N,d}\to \Lambda_{N,N-d}$.
\end{proposition}

As a map of eigenstep tableaux, $\Psi_{N,d}$ can be understood as interchanging rows and diagonals in the parallelogram of non-fixed entries while adjusting the sizes of $0$- and $N$-triangles. For example, when $N=5$, $d=3$ the map $\Psi_{5,3}\colon \aff(\Lambda_{5,3})\to\aff(\Lambda_{5,2})$ is given by
\begin{equation*}
    \Psi_{5,3}
    \begin{pmatrix}
        0 & 0 & 0 & \lambda_{3,3} & \lambda_{3,4} & 5 \\
        0 & 0 & \lambda_{2,2} & \lambda_{2,3} & 5 & 5 \\
        0 & \lambda_{1,1} & \lambda_{1,2} & 5 & 5 & 5
    \end{pmatrix}
    =
    \begin{pmatrix}
        0 & 0 & \lambda_{3,3} & \lambda_{2,2} & \lambda_{1,1} & 5 \\
        0 & \lambda_{3,4} & \lambda_{2,3} & \lambda_{1,2} & 5 & 5
    \end{pmatrix}.
\end{equation*}

In \Cref{fig:psi} we illustrate the general structure of the image of an eigenstep tableau under $\Psi_{N,d}$.

\begin{figure}[h]
    \centering
    \includegraphics{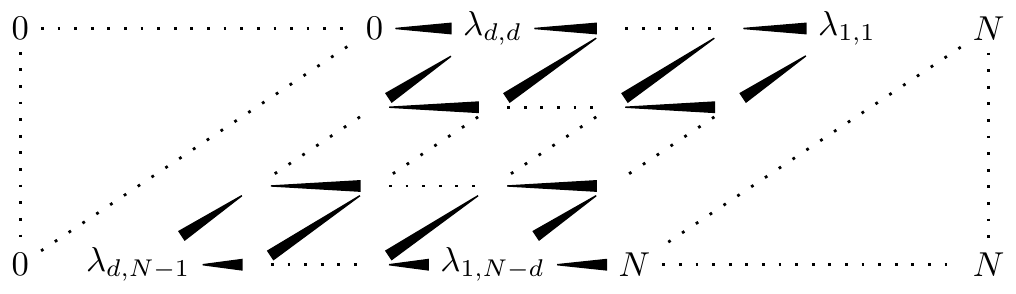}
    \caption{The image of a sequence of eigensteps $\lambda$ as in \Cref{fig:eigenstep-triangles} under the affine isomorphism $\Psi_{N,d}\colon \Lambda_{N,d}\to\Lambda_{N,N-d}$ from \Cref{prop:psi}.}
    \label{fig:psi}
\end{figure}

\begin{proof}
    We first need to verify that $\lambda' \coloneqq \Psi_{N,d}(\lambda)$ is a point in $\aff(\Lambda_{N,N-d})$ for $\lambda\in\aff(\Lambda_{N,d})$. The triangle conditions are satisfied by the definition of $\Psi_{N,d}$. To verify the sum conditions for $\Lambda_{N,N-d}$, let $1\le n \le N-1$ and $m\coloneqq N-n$, then
    \begin{align*}
        \sum_{i=1}^{N-d} \lambda'_{i,n} &=
        \smashoperator[r]{\sum_{i=1}^{\max\{0,n-d\}}} N \quad +
        \smashoperator[r]{\sum_{i=\max\{0,n-d\}+1}^{\min\{N-d,n\}}} \lambda_{d+i-n,N-n} \\
        &= \max\{0,n-d\} N + \smashoperator[r]{\sum_{j=\max\{0,m+d-N\}+1}^{\min\{d,m\}}} \lambda_{j,m} \\
        &= \max\{0,n-d\} N + dm - \max\{0,m+d-N\} N \\
        &= \max\{0,n-d\} N - \max\{0,d-n\} N + d(N-n) \\
        &= (n-d)N + d(N-n) = (N-d)n.
    \end{align*}
    To see that $\Psi_{N,d}$ restricts to an affine map $\Lambda_{N,d}\to\Lambda_{N,N-d}$ we need to consider all inequalities. Let $\lambda\in\Lambda_{N,d}$ and $\lambda'=\Psi_{N,d}(\lambda)$. The lower and upper bound conditions are satisfied, since $\lambda'_{N-d,N-d} = \lambda_{d,d}\ge 0$ and $\lambda'_{1,d} = \lambda_{1,N-d} \le N$. The remaining horizontal and diagonal inequalities \eqref{ineq:prop1-horiz} and \eqref{ineq:prop1-diag} interchange under $\Psi_{N,d}$. Let $j\coloneqq d+i-n$ and $m\coloneqq N-n$, then we have
    \begin{align*}
        \lambda'_{i,n} &\le \lambda'_{i,n+1}
        & \text{for} \quad & 1\le i \le N-d,\enskip i\le n < d+i-1\\
        \Leftrightarrow\quad \lambda_{j,m} &\le \lambda_{j-1,m-1}
        & \text{for} \quad & 1< j\le d,\enskip j \le m < N-d+j\\
        \intertext{and}
        \lambda'_{i,n} &\le \lambda'_{i-1,n-1}
        & \text{for} \quad & 1< i\le N-d,\enskip i \le n < d+i \\
        \Leftrightarrow\quad \lambda_{j,m} &\le \lambda_{j,m+1}
        & \text{for} \quad & 1\le j \le d,\enskip j\le m < N-d+j-1.
    \end{align*}
    Hence, $\Psi_{N,d}$ restricts to an affine map $\Lambda_{N,d}\longrightarrow\Lambda_{N,N-d}$. It is an isomorphism on both the affine hulls and the polytopes themselves, since $\Psi_{N,d}$ and $\Psi_{N,N-d}$ are mutually inverse. This needs to be checked only for the non-fixed entries:
    \begin{align*}
        \left(\Psi_{N,N-d}(\Psi_{N,d}(\lambda))\right)_{i,n} &=
        \left(\Psi_{N,d}(\lambda)\right)_{N-d+i-n, N-n} \\&=
        \lambda_{d+N-d+i-n-N+n, N-N+n} = \lambda_{i,n}, \\
        \left(\Psi_{N,d}(\Psi_{N,N-d}(\lambda))\right)_{i,n} &=
        \left(\Psi_{N,N-d}(\lambda)\right)_{d+i-n,N-n} \\&=
        \lambda_{N-d+d+i-n-N+n, N-N+n} = \lambda_{i,n}.\qedhere
    \end{align*}
\end{proof}

\begin{proposition} \label{prop:phi}
    There is an affine involution $\Phi_{N,d}\colon \R^{d\times(N+1)}\longrightarrow \R^{d\times(N+1)}$ given by
    \begin{equation*}
        \Phi(\lambda)_{i,n} = N - \lambda_{d-i+1, N-n},
    \end{equation*}
    that restricts to an affine involution $\Lambda_{N,d}\to\Lambda_{N,d}$.
\end{proposition}

The involution $\Phi_{N,d}$ can be described as rotating the whole eigenstep tableau by $180\degree$ and subtracting every entry from $N$, as depicted in \Cref{fig:phi}.

\begin{figure}[b]
    \centering
	\includegraphics{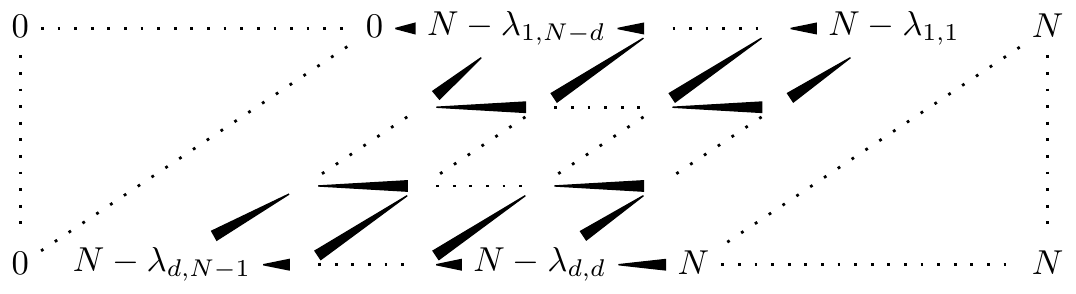}
    \caption{The image of a sequence of eigensteps $\lambda$ as in \Cref{fig:eigenstep-triangles} under the involution $\Phi_{N,d}$ from \Cref{prop:phi}.}
    \label{fig:phi}
\end{figure}

\begin{proof}
    It is clear that $\Phi_{N,d}$ is an affine map $\R^{d\times(N+1)}\to\R^{d\times(N+1)}$. We use the original system of equations and inequalities given in \Cref{def:eigensteps} to verify $\Phi(\lambda)\in \Lambda_{N,d}$ when $\lambda\in\Lambda_{N,d}$. For $n=0, N$ we obtain
    \begin{align*}
        \Phi(\lambda)_{i,0} = N - \lambda_{d-i+1, N} = N - N = 0, \\
        \Phi(\lambda)_{i,N} = N - \lambda_{d-i+1, 0} = N - 0 = N.
    \end{align*}
    Hence \eqref{eq:def-0} and \eqref{eq:def-N} are satisfied by $\lambda'$. The column sum conditions \eqref{eq:def-sums} are satisfied, since
    \begin{align*}
        \sum_{i=1}^d \Phi(\lambda)_{i,n} &= \sum_{i=1}^d \left(N - \lambda_{d-i+1, N-n}\right)\\
        &= \sum_{j=1}^d \left(N - \lambda_{j, N-n}\right)\\
        &= dN - d(N-n)\\
        &= dn.
    \end{align*}
    For the horizontal and diagonal inequalities, we observe that $\lambda_{i,n} \le \lambda_{i',n'}$ is equivalent to $N - \lambda_{i',n'} \le N - \lambda_{i,n}$.

    Finally, $\Phi_{N,d}$ is an involution on both $\R^{d\times(N+1)}$ and $\Lambda_{N,d}$, since
    \begin{align*}
        \big( (\Phi_{N,d} \circ \Phi_{N,d})(\lambda)\big)_{i,n} &= N - \big(\Phi_{N,d}(\lambda)\big)_{d-i+1,N-n}\\
        &= N - (N - \lambda_{i,n})\\
        &= \lambda_{i,n}.\qedhere
    \end{align*}
\end{proof}

The results noted in the following remark are easily verified by direct computation.

\begin{remark}
    The special point $\widehat\lambda$ of $\Lambda_{N,d}$ is fixed under $\Phi_{N,d}$ and mapped to the special point of $\Lambda_{N,N-d}$ by $\Psi_{N,d}$. Furthermore, $\Phi$ and $\Psi$ commute. To be precise:
    \begin{equation*}
        \Phi_{N,N-d} \circ \Psi_{N,d} = \Psi_{N,d} \circ \Phi_{N,d}.
    \end{equation*}
\end{remark}

Using the dualities given by $\Phi$ and $\Psi$, we now construct points that witness the necessity of most of the inequalities in \Cref{prop:eigensteps}.

\begin{lemma} \label{lem:failineq}
    Let $N\ge 5$, $2\le d\le N-2$. Consider one of the inequalities in \eqref{ineq:prop1-horiz} to \eqref{ineq:prop1-ub} which is not $\lambda_{2,2} \le \lambda_{1,1}$, $\lambda_{1,1}\le\lambda_{1,2}$, $\lambda_{d,N-2}\le\lambda_{d,N-1}$ or $\lambda_{d,N-1}\le\lambda_{d-1,N-2}$. Then there is a point in $\R^{d\times(N+1)}$ satisfying all conditions of \Cref{prop:eigensteps} except the considered inequality. 
\end{lemma}

\begin{proof}
    The idea behind the proof is to start with the special point $\widehat\lambda \in \Lambda_{N,d}$ and locally change entries such that just one of the inequalities fails, while preserving all other conditions. Since $\Psi_{N,d}$ translates horizontal \eqref{ineq:prop1-horiz} to diagonal inequalities \eqref{ineq:prop1-diag} and vice versa, it is enough to consider only horizontal inequalities. Also, since $\Phi_{N,d}$ maps the top row ($i=d)$ to the bottom row ($i=1$), the inequalities in the bottom row do not need to be considered either. Since $\Phi_{N,d}$ maps the first diagonal ($i=n$) to the last ($i=n+d-N+1$) and vice versa, we do not need to consider the last horizontal inequality in each row. The remaining horizontal inequalities are treated with the following modification of $\widehat\lambda$:
    \begin{equation}
        \includegraphics{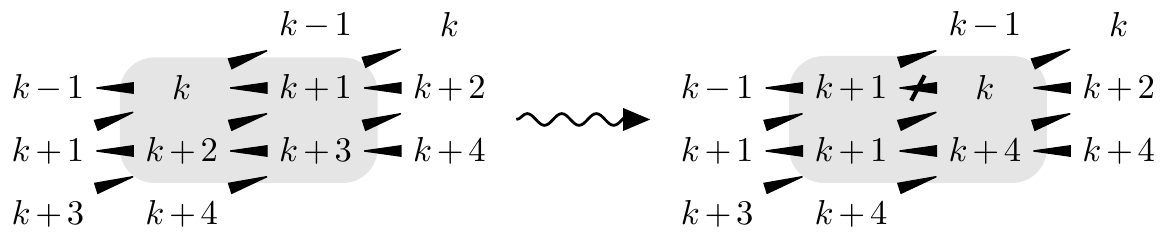}
        \label{eq:move-horiz}
    \end{equation}
    Note that this modification of $\widehat\lambda$ does not alter the column sums and causes only the slashed inequality in \eqref{eq:move-horiz} to fail.
    If $2\le d<N-2$ the square of modified entries in \eqref{eq:move-horiz} fits into the parallelogram of non-fixed entries.
\begin{figure}
    \centering
    \includegraphics{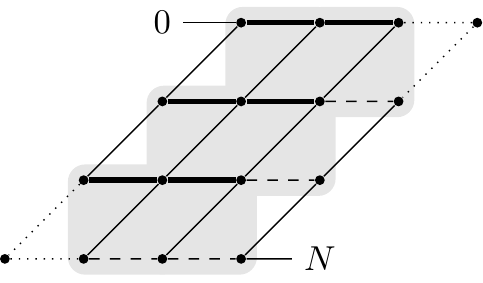}
    \caption{The horizontal inequalities treated by the modification \eqref{eq:move-horiz} shown in bold.}
    \label{fig:horiz-cover}
\end{figure}
In \Cref{fig:horiz-cover} we demonstrate how this modification can be used to obtain points that let each of the bold inequalities fail individually. The dashed inequalities are covered by the above argument using $\Phi_{N,d}$, while the dotted inequalities are the four exceptions mentioned in \Cref{lem:failineq}.

    If $d=N-2$, the parallelogram of non-fixed entries becomes too thin to fit the squares of \eqref{eq:move-horiz}, so this case has to be treated separately. Instead of considering the horizontal inequalities for $d=N-2$, we can use the duality given by $\Psi_{N,d}$ and consider the diagonal inequalities for $d=2$. We use the following modification of $\widehat\lambda$:
    \begin{equation*}
        \includegraphics{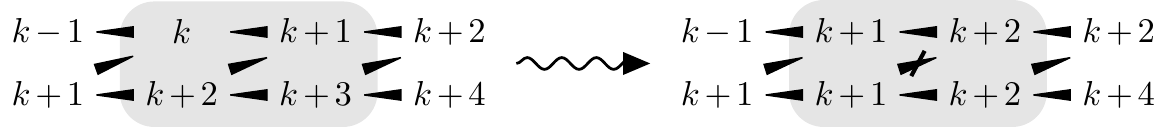}
    \end{equation*}

    The only inequality that remains to be treated is the lower bound condition $\lambda_{d,d}\ge 0$. The upper bound condition then follows from the duality given by $\Phi_{N,d}$. Here we use a modification of $\widehat\lambda$ to construct a point that causes only the lower bound condition to fail. We first do this for $d=2$:
    \begin{equation*}
        \includegraphics{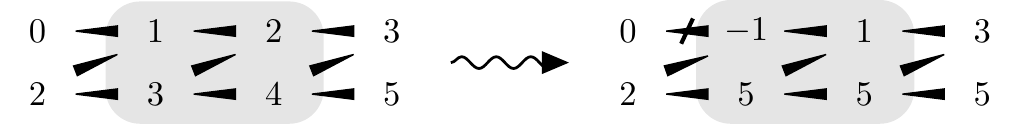}
    \end{equation*}
    This also covers the case $d=N-2$ by dualizing using $\Psi_{N,2}$.
    For $2<d<N-2$ we use a different modification of $\widehat\lambda$:
    \begin{equation*}
        \includegraphics{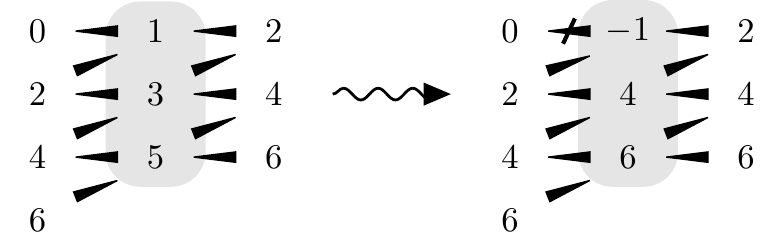}\qedhere
    \end{equation*}
\end{proof}

\begin{example} \label{ex:52}
    For $N=5$ and $d=2$, we construct the points given by \Cref{lem:failineq} explicitly. The special point of $\Lambda_{5,2}$ is
    \begin{equation*}
        \widehat\lambda =
        \begin{pmatrix}
            0 & 0 & 1 & 2 & 3 & 5 \\
            0 & 2 & 3 & 4 & 5 & 5
        \end{pmatrix}.
    \end{equation*}
    The half-spaces described by the non-exceptional inequalities are
        $H_1\colon \lambda_{2,2} \ge 0$,
        $H_2\colon \lambda_{2,2} \le \lambda_{2,3}$,
        $H_3\colon \lambda_{2,3} \le \lambda_{1,2}$,
        $H_4\colon \lambda_{1,2} \le \lambda_{1,3}$ and
        $H_5\colon \lambda_{1,3} \le 5$.
    Applying \Cref{lem:failineq} yields the following five points $P_i$, each satisfying all conditions except lying in the half-space $H_i$:
    \begin{align*}
        P_1 &=
        \begin{pmatrix}
            0 & 0 & -1 & 1 & 3 & 5 \\
            0 & 2 & 5 & 5 & 5 & 5
        \end{pmatrix}, \\
        P_2 &=
        \begin{pmatrix}
            0 & 0 & 2 & 1 & 3 & 5 \\
            0 & 2 & 2 & 5 & 5 & 5
        \end{pmatrix}, \\
        P_3 &=
        \begin{pmatrix}
            0 & 0 & 2 & 3 & 3 & 5 \\
            0 & 2 & 2 & 3 & 5 & 5
        \end{pmatrix}, \\
        P_4 &= \Phi_{5,2}(P_2) =
        \begin{pmatrix}
            0 & 0 & 0 & 3 & 3 & 5 \\
            0 & 2 & 4 & 3 & 5 & 5
        \end{pmatrix}, \\
        P_5 &= \Phi_{5,2}(P_1) =
        \begin{pmatrix}
            0 & 0 & 0 & 0 & 3 & 5 \\
            0 & 2 & 4 & 6 & 5 & 5
        \end{pmatrix}.
    \end{align*}
    The two variables $\lambda_{2,2}$ and $\lambda_{2,3}$ completely parametrize the polytope, since $\lambda_{1,1}=2$, $\lambda_{1,2} = 4-\lambda_{2,2}$, $\lambda_{1,3} = 6-\lambda_{2,3}$ and $\lambda_{2,4}=3$ by the column sum conditions. Hence, we can illustrate the situation in the plane, as done in \Cref{fig:polytope_5_2}.
    \begin{figure}
        \centering
        \includegraphics{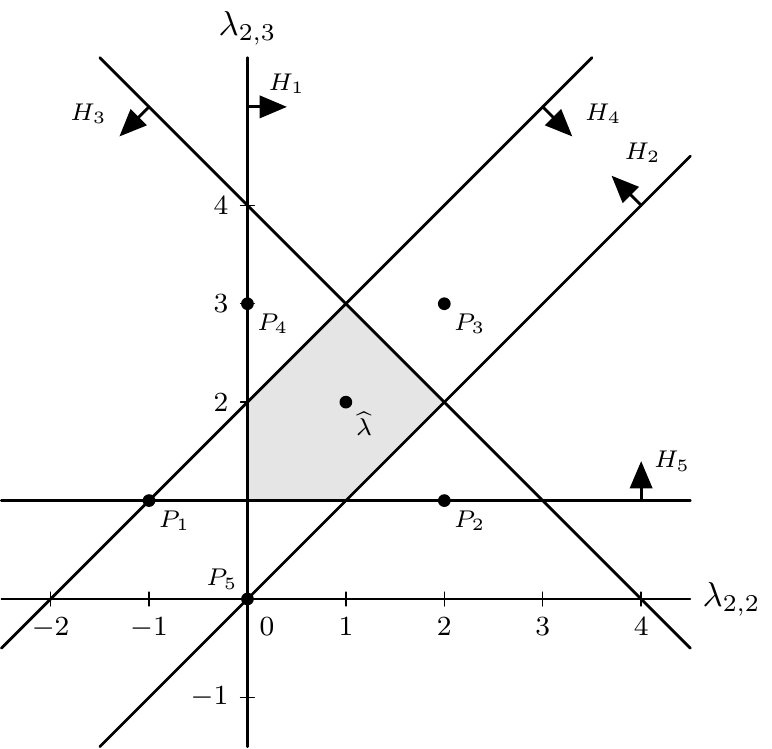}
        \caption{For $\Lambda_{5,2}$ we have five necessary inequalities. The points $P_i$ which satisfy all conditions but the defining inequality for the half-space $H_i$ are constructed in \Cref{ex:52}.}
        \label{fig:polytope_5_2}
    \end{figure}
\end{example}

Using \Cref{lem:failineq}, we prove the following theorem, giving the number of facets of $\Lambda_{N,d}$.

\begin{theorem}
    \label{thm:nfacets}
    For $2\le d\le N-2$ the number of facets of $\Lambda_{N,d}$ is
    \begin{equation*}
        d(N-d-1)+(N-d)(d-1)-2.
    \end{equation*}
\end{theorem}

\begin{proof}
    We first show for the case of $N\geq 5$ that $d(N-d-1)+(N-d)(d-1)-2$ inequalities are sufficient to describe $\Lambda_{N,d}$ in its affine hull.

    Let $N\geq 5$, $2\leq d\leq N-2$.
    Counting the horizontal and diagonal inequalities \eqref{ineq:prop1-horiz}, \eqref{ineq:prop1-diag} yields $d(N-d-1) + (d-1)(N-d)$ inequalities.

    We now show that the four inequalities between non-fixed entries that are already mentioned in \Cref{lem:failineq} are in fact not necessary. Recall that these are
    \begin{align}
            \lambda_{2,2}&\leq \lambda_{1,1}, \label{eq:a1}\\
            \lambda_{1,1}&\leq \lambda_{1,2}, \label{eq:a2}\\
            \lambda_{d,N-2}&\leq \lambda_{d,N-1}, \label{eq:a3}\\
            \lambda_{d,N-1}&\leq \lambda_{d-1,N-2}. \label{eq:a4}
    \end{align}
    From the column sum and triangle conditions it follows that $\lambda_{1,1} = d$ and $\lambda_{1,2} + \lambda_{2,2} = 2d$. 
    Thus \eqref{eq:a1} and \eqref{eq:a2} are both equivalent to $\lambda_{2,2}\le d$, which is already implied by $\lambda_{2,2} \le \lambda_{2,3} \le \lambda_{1,2} = 2d-\lambda_{2,2}$, when $d\le N-2$. Therefore \eqref{eq:a1} and \eqref{eq:a2} are superfluous.
    
    Again, by the column sum and triangle conditions, we have $\lambda_{d,N-1} = N-d$ and $\lambda_{d,N-2}+\lambda_{d-1,N-2}=2(N-d)$. Thus \eqref{eq:a3} and \eqref{eq:a4} are both equivalent to $\lambda_{d-1,N-2}\ge N-d$, which is already implied by $\lambda_{d-1,N-2}\ge \lambda_{d-1,N-3} \ge \lambda_{d,N-2} = 2(N-d)-\lambda_{d-1,N-2}$, when $d\le N-2$. However the two arguments are independent only when $N\ge 5$, since for $N=4$, $d=2$ we have $\lambda_{2,2} = \lambda_{d,N-2}$ and $\lambda_{1,2} = \lambda_{d-1,N-2}$.

    Counting all inequalities, including the lower and upper bound conditions, excluding the four superfluous inequalities, we have
    \begin{equation*}
        d(N-d-1) + (d-1)(N-d) + 2 - 4 = d(N-d-1) + (d-1)(N-d) - 2
    \end{equation*}
    inequalities that are sufficient to describe $\Lambda_{N,d}$. From \Cref{lem:failineq} we know that all these inequalities are actually necessary, hence we obtain the desired number of facets.

    For the case $N=4$, $d=2$, we have $\dim(\Lambda_{4,2})=(2-1)(4-2-1)=1$. The only polytope of dimension $1$ is a line segment, the two endpoints being its facets. Thus, $\Lambda_{4,2}$ has two facets, as given by $d(N-d-1)+(d-1)(N-d)-2$ for $N=4$, $d=2$.
\end{proof}

From \Cref{lem:failineq} and \Cref{thm:nfacets} we conclude that removing the four exceptional inequalities from the description of $\Lambda_{N,d}$ in \Cref{prop:eigensteps} yields a non-redundant system of equations and inequalities.

\section{Connections between frame and eigenstep operations}
\label{sec:frame-connections}

Until now we focused on the combinatorics of sequences of eigensteps. In this section, we give descriptions of the affine isomorphisms $\Phi_{N,d}$ and $\Psi_{N,d}$ in terms of the underlying frames.

For this section, we fix the following notations: given a frame $F=(f_n)_{n=1}^N$, let $\lambda_F$ denote the sequence of eigensteps associated to an equal norm tight frame $F$, that is $\lambda_F \coloneqq (\sigma(F_n^{\vphantom{*}} F_n^*))_{n=0}^N$, and let $\widetilde F \coloneqq \left(f_{N-n+1}\right)_{n=1}^N$ denote the frame with reversed order of frame vectors.

We obtain the following result:
\begin{proposition}
    Let $F=(f_n)_{n=1}^N$ be an equal norm tight frame in $\F^d$ with $\|f_n\|^2 = d$, then
    \begin{equation*}
        \Phi_{N,d}(\lambda_F) = \lambda_{\widetilde F}.
    \end{equation*}
\end{proposition}

\begin{proof}
    Decomposing the frame operator of $F$ we have
    \begin{equation*}
        N\cdot I_d = FF^* = \sum_{k=1}^N f_k^{\vphantom{*}} f_k^* = \sum_{k=1}^n f_k^{\vphantom{*}} f_k^* + \sum_{k=n+1}^N f_k^{\vphantom{*}} f_k^* = F_n^{\vphantom{*}} F_n^* + \widetilde F_{N-n}^{\vphantom{*}} \widetilde F_{N-n}^*.
    \end{equation*}
    Thus, if $v\in\F^d$ is an eigenvector of $F_n^{\vphantom{*}} F_n^*$ with eigenvalue $\gamma$, we obtain
    \begin{equation*}
        \widetilde F_{N-n}^{\vphantom{*}} \widetilde F_{N-n}^* v = (N\cdot I_d-F_n^{\vphantom{*}} F_n^*) v = (N-\gamma) v.
    \end{equation*}
    So $v$ is an eigenvector of $\widetilde F_{N-n}^{\vphantom{*}} \widetilde F_{N-n}^*$ with eigenvalue $N-\gamma$ and $\lambda_{\widetilde F} = \Phi_{N,d}(\lambda_F)$.
\end{proof}

A well-known concept in finite frame theory is the notion of \emph{Naimark complements}. In the case of Parseval frames, finding a Naimark complement of $F$ amounts to finding a matrix $G$ such that $\bigl(\begin{smallmatrix}F\\G\end{smallmatrix}\bigr)$ is unitary. By scaling, this definition can be extended to tight frames and in fact to all finite frames, as discussed in \cite{CFM11b}. In our context, we use the following definition:

\begin{definition}\label{def:naimark}
    Given an equal norm tight frame $F =(f_n)_{n=1}^N$ in $\F^d$ with $\|f_n\|^2 = d$, a frame $G = (g_n)_{n=1}^N$ in $\F^{N-d}$ satisfying $F^*F + G^*G = N\cdot I_N$ is called a \emph{Naimark complement} of $F$.
\end{definition}
Many properties of a frame $F$ carry over to its Naimark complement $G$. In particular, a Naimark complement of an equal norm tight frame is again an equal norm tight frame, the norm being $\sqrt{N-d}$. The following proposition shows how the duality described by $\Psi_{N,d}$ corresponds to taking a Naimark complement and reversing the order of frame vectors.

\begin{proposition}
    Let $F=(f_n)_{n=1}^N$ be an equal norm tight frame in $\F^d$ with norms $\|f_n\|^2 = d$ and $G=(g_n)_{n=1}^N$ a Naimark complement of $F$, then
    \begin{equation*}
        \Psi_{N,d}(\lambda_F) = \lambda_{\widetilde G}.
    \end{equation*}
\end{proposition}

\begin{proof}
    Since $\Psi_{N,d}(\lambda_F) = \lambda_{\widetilde G}$ is equivalent to $\lambda_F = \Psi_{N,N-d}(\lambda_{\widetilde G})$, we only need to consider the case $N\ge 2d$.
    We first consider the columns of $F$ with indices $n<d$. Since $F_n$ is an $d\times n$ matrix, $F_n^{\vphantom{*}} F_n^*$ has at most $n$ non-zero eigenvalues. To be precise, the spectrum of the frame operator of $F_n$ is
    \begin{equation*}
        \sigma(F_n^{\vphantom{*}} F_n^*) = (\lambda_{1,n},\dots,\lambda_{n,n},\underbrace{0,\dots,0}_{d-n}).
    \end{equation*}
    In order to obtain the eigensteps of $G$, we switch to Gram matrices. The Gram matrix of $F_n$ is the $n\times n$ matrix $F_n^* F_n^{\vphantom{*}}$, with spectrum
    \begin{equation*}
        \sigma(F_n^* F_n^{\vphantom{*}}) = (\lambda_{1,n},\dots,\lambda_{n,n}),
    \end{equation*}
    which is obtained by considering the singular value decomposition of $F_n$.

    Since $G$ is a Naimark complement of $F$, we have $F^*F + G^*G = N\cdot I_N$. In particular,
    \begin{align*}
        N\cdot I_N &= F^*F + G^*G =
        \begin{pmatrix}
            F^* & G^*
        \end{pmatrix}
        \begin{pmatrix}
            F \\ G
        \end{pmatrix}
        \\&=
        \begin{pmatrix}
            F_n^* & G_n^* \\
            \vdots & \vdots
        \end{pmatrix}
        \begin{pmatrix}
            F_n & \cdots \\
            G_n & \cdots
        \end{pmatrix}
        =
        \begin{pmatrix}
            F_n^* F_n^{\vphantom{*}} + G_n^* G_n^{\vphantom{*}} & \cdots \\
            \vdots & \ddots
        \end{pmatrix}.
    \end{align*}
    The first $n$ rows and columns of this identity yield $F_n^* F_n^{\vphantom{*}} + G_n^* G_n^{\vphantom{*}} = N \cdot I_n$. Therefore
    \begin{equation*}
        \sigma(G_n^* G_n^{\vphantom{*}}) = (N-\lambda_{n,n},\dots,N-\lambda_{1,n}).
    \end{equation*}
    Going back to the frame operator of $G_n$, which is the $(N-d)\times(N-d)$ matrix $G_n^{\vphantom{*}} G_n^*$, we have
    \begin{equation*}
        \sigma(G_n^{\vphantom{*}} G_n^*) = (N-\lambda_{n,n},\dots,N-\lambda_{1,n},\underbrace{0,\dots,0}_{N-d-n}).
    \end{equation*}
    Finally, using $\widetilde G_{N-n}^{\vphantom{*}} \widetilde G_{N-n}^* + G_n^{\vphantom{*}} G_n^* = GG^* = N\cdot I_{N-d}$, we obtain
    \begin{equation*}
        \sigma(\widetilde G_{N-n}^{\vphantom{*}} \widetilde G_{N-n}^*) = (\underbrace{N,\dots,N}_{N-d-n},\lambda_{n,n},\dots,\lambda_{1,n}),
    \end{equation*}
    which shows that the $(N-n)$-th column of $\Psi(\lambda_F)$ is equal to the $(N-n)$-th column of $\lambda_{\widetilde G}$ for $n<d$.

    For $n>N-d$, let $m\coloneqq N-n$ so that $m < d$. Hence, the $(N-m)$-th column of $\Psi(\lambda_{\widetilde F})$ is the $(N-m)$-th column of $\lambda_G$ by the previous argument. Since $\lambda_{\widetilde F} = \Phi_{N,d}(\lambda_F)$ and $\lambda_G = \Phi_{N,N-d}(\lambda_{\widetilde G})$, we know that $\Psi_{N,d}(\Phi_{N,d}(\lambda_F))$ and $\Phi_{N,N-d}(\lambda_{\widetilde G})$ agree in the $n$-th column. Using $\Psi_{N,d}\circ \Phi_{N,d} = \Phi_{N,N-d} \circ \Psi_{N,d}$ and the fact that $\Phi_{N,N-d}$ reverses the column order, we conclude that $\Psi_{N,d}(\lambda_F)$ and $\lambda_{\widetilde G}$ agree in the $(N-n)$-th column as desired.

    We now consider $d\le n \le N-d$. By the same arguments as before, we have
    \begin{align*}
        \sigma(F_n^{\vphantom{*}} F_n^*) &= (\lambda_{1,n},\dots,\lambda_{d,n}), \vphantom{\underbrace{}_{n-d}}\\
        \sigma(F_n^* F_n^{\vphantom{*}}) &= (\lambda_{1,n},\dots,\lambda_{d,n}, \underbrace{0,\dots,0}_{n-d}), \\
        \sigma(G_n^* G_n^{\vphantom{*}}) &= (\underbrace{N,\dots,N}_{n-d},N-\lambda_{d,n},\dots,N-\lambda_{1,n}).
    \end{align*}
    Since $G_n$ is an $(N-d)\times n$ matrix, with $N-d \ge n$, the spectrum of the frame operator of $G_n$ is
    \begin{equation*}
        \sigma(G_n^{\vphantom{*}} G_n^*) = (\underbrace{N,\dots,N}_{n-d},N-\lambda_{d,n},\dots,N-\lambda_{1,n}, \underbrace{0,\dots,0}_{N-d-n}),
    \end{equation*}
    thus
    \begin{equation*}
        \sigma(\widetilde G_{N-n}^{\vphantom{*}} \widetilde G_{N-n}^*) = (\underbrace{N,\dots,N}_{N-d-n},\lambda_{1,n},\dots,\lambda_{d,n}, \underbrace{0,\dots,0}_{n-d}),
    \end{equation*}
    which shows that the $(N-n)$-th column of $\Psi(\lambda_F)$ is equal to the $(N-n)$-th column of $\lambda_{\widetilde G}$ for $d\le n\le N-d$.
\end{proof}

\section{Conclusion and open problems}
\label{sec:conclusion}

As we have seen, in the special case of equal norm tight frames we are able to obtain a general non-redundant description of the polytope of eigensteps in terms of equations and inequalities. However, this description does not generalize to non-tight frames, where we lose the $N$-triangle in the eigenstep tableau. Hence, even the dimension of $\Lambda((\mu_n)_{n=1}^N, (\lambda_i)_{i=1}^d)$ will depend on the multiplicities of eigenvalues in the spectrum that cause smaller triangles of fixed entries in the eigenstep tableaux.

From a discrete geometers point of view, it might be interesting to find a description of polytopes of eigensteps in terms of vertices. However, even restricting to equal norm tight frames, we were not able to calculate the number of vertices of $\Lambda_{N,d}$ in general, let alone find a description of the polytope as a convex hull of vertices. On the frame theoretic end, it might be interesting to study properties of frames $F$ corresponding to certain points of the polytope. For example, interesting classes of equal norm tight frames might be the frames $F$ such that $\lambda_F$ is the special point $\widehat\lambda$, a boundary point of $\Lambda_{N,d}$ or a vertex of $\Lambda_{N,d}$.

\printbibliography
\end{document}